\documentclass[paper=a4, fontsize=11pt]{scrartcl} 

\usepackage[T1]{fontenc} 
\usepackage{fourier} 
\usepackage[english]{babel} 
\usepackage{amsmath,amsfonts,amsthm} 
\usepackage{amsmath, calligra, mathrsfs}
\usepackage{amssymb}
\usepackage{mathtools}
\usepackage{hyperref}
\usepackage{mathdots}
\usepackage{array}
\usepackage[mathscr]{euscript}
\usepackage{verbatim}
\usepackage[dvipsnames]{xcolor}
\usepackage{mdframed} 
\usepackage{soulutf8}
\usepackage{stmaryrd}
\usepackage{tikz-cd}
\usepackage{hyperref}
\usepackage{lipsum} 

\usepackage{sectsty} 
\allsectionsfont{\centering \normalfont\scshape} 

\usepackage{fancyhdr} 
\usepackage{xurl}
\newcommand*{\sheafhom}{\mathcal{H}\kern -.5pt om}
\numberwithin{equation}{section} 
\numberwithin{figure}{section} 
\numberwithin{table}{section} 

\newtheorem{thm}{Theorem}[section]
\newtheorem{cor}[thm]{Corollary}
\newtheorem{prop}[thm]{Proposition}

\newtheorem{lem}[thm]{Lemma}

\theoremstyle{definition}
\newtheorem{defn}[thm]{Definition}

\newtheorem{exmp}[thm]{Example}

\theoremstyle{remark}
\newtheorem{rem}[thm]{Remark}

\DeclareMathOperator{\St}{st}

\DeclareMathOperator{\lk}{lk}

\DeclareMathOperator{\sgn}{sgn}

\DeclareMathOperator{\Ast}{ast}
\DeclareMathOperator{\rk}{rank}

\newcommand{\horrule}[1]{\rule{\linewidth}{#1}} 

\title{	
	\normalfont \normalsize 
	\textsc{} \\ [25pt] 
	\horrule{0.5pt} \\[0.4cm] 
	\huge Flag complex face structures and decompositions

	\horrule{2pt} \\[0.5cm] 
}

\author{Soohyun Park } 

\date{\normalsize November 3, 2025} 

\begin{document}
	
	\maketitle 

	\begin{abstract}
		
		\noindent One of the most common and effective methods of obtaining structural information on simplicial complexes is to use tools from algebraic geometry/commutative algebra (often motivated by properties of toric varieties). However, there is no general algebro-geometric description of components of the gamma vector holding for arbitary flag simplicial spheres. This invariant occurs in many different contexts including permutation statistics, signatures of toric varieties, and Euler characteristics of nonpositively curved piecewise Euclidean manifolds. Combinatorial methods resulting from an explicit inverted Chebyshev expansion give rise to new positivity properties and cell complex structures that are of interest in their own right. Note that the focus is on the $f$-vector rather than the $h$-vector in ``algebraic'' settings. For flag simplicial spheres $\Delta$, the fact that $h(\Delta) = f(\Gamma)$  and compatibility between Chebyshev expansions and a modification of the $f$-polynomial by work of Hetyei are the key inputs. In the main formula implying new positivity results, local structures of $CAT(0)$ complexes and cubical analogues of barycentric subdivisions give deeper connections with cubical complex structures complementing earlier work related to the top gamma vector component. Afterwards, we return to the motivating example of barycentric subdivisions and consider how $f$-vectors of Cohen--Macaulay and vertex decomposable flag complexes in geometric settings decompose and interact with geometric transformations. This includes subdivisions of simplicial complexes and recursive properties they share with vertex decomposable flag complexes.
		
	\end{abstract}
	
	\section*{Introduction}

	In general, structural information on faces of simplicial complexes is often obtained via tools from commutative algebra and algebraic geometry to extract algebraic/geometric information on simplicial complex structures. This makes use of the $h$-vector (analogous to the even degree Betti numbers of simplicial toric varieties) instead of the $f$-vector that literally counts the number of faces of each dimension. Note that the two are technically related by an invertible linear transformation (p. 213 of \cite{BH}). This has been useful in many settings such as the proof of the upper bound theorem (Corollary 5.3 on p. 141 of \cite{Stubc}, Corollary 3.5 on p. 60 of \cite{St}, Corollary 5.4.7 on p. 240 of \cite{BH}, survey in \cite{FMS}). However, our main object of interest is a vector constructed out of the $h$-vector where a suitable description only exists for one component. \\
	
	To be more specific, this is the gamma vector of a reciprocal/palindromic polynomial. It appears in a wide range of combinatorial and geometric contexts including permutation statistics, signatures of toric varieties, and Euler characteristics of nonpositively curved piecewise Euclidean manifolds \cite{Athgam}. The nonnegativity of the gamma vector lies somewhere between unimodality and real-rootedness. Many of these examples fall under the setting of flag simplicial spheres. These are simplicial spheres where the 1-skeleton is the clique complex of a graph (equivalently having minimal nonfaces of size 2, see p. 100 of \cite{St}). The top component is the only component where there is a large class of cases with an algebraic interpretation (as the signature of a toric variety by work of Leung--Reiner \cite{LR}). We also note that the top component is the source of connections with nonpositively curved piecewise Euclidean manifolds by work of Charney--Davis \cite{CD}. \\

	As it turns out, a combinatorial view towards studying the gamma vector as a whole gives novel structural information of its own right. This involves both the cell complex structure and new positivity properties. Note that the $f$-vector plays a larger role here than the $h$-vector in the context of the geometric methods mentioned at the beginning. Our approach started in \cite{Pgamcheb} with looking at explicit formulas via gamma vector-like invariants for ``usual'' basis $x^k + x^{d - k}$ of palindromic/reciprocal polynomials of degree $\le d$. This led to an interpretation of the reciprocal polynomial of the gamma polynomial as a sort of inverted Chebyshev expansion (Theorem \ref{gamchebinv}).

	\begin{thm} \textbf{(Gamma vectors and inverted Chebyshev expansions, Theorem 1.7 on p. 7 of \cite{Pgamcheb})} \label{gamchebinv} \\
		Let $h(t) = h_0 + h_1 t + \ldots + h_{d - 1} t^{d - 1} + h_d t^d$ be a reciprocal polynomial satisfying the relations $h_k = h_{d - k}$ for $0 \le k \le \frac{d}{2}$. Then, the gamma-polynomial associated to $h(t)$ is equal to \[ \gamma(u) = u^{ \frac{d}{2} } g \left( \frac{1}{u} - 2 \right), \] where \[ g(u) \coloneq h_{ \frac{d}{2} } + 2 \sum_{ j = 1}^{ \frac{d}{2} } h_{ \frac{d}{2} - j } T_j \left( \frac{u}{2} \right). \] This is a sort of ``inverted Chebyshev basis expansion''. \\
		
		Alternatively, we can rewrite this as \[ u^{ \frac{d}{2}  }  \gamma \left( \frac{1}{u} \right)  = g(u - 2) \] or \[ (u + 2)^{ \frac{d}{2} } \gamma \left( \frac{1}{u + 2} \right) = g(u). \]
	\end{thm}

	In addition to the properties studied in \cite{Pgamcheb}, the inverted Chebyshev expansion perspective also yields information on cell complex structures. This comes from the fact that $h(\Delta) = f(\Gamma)$ for an auxiliary balanced simplicial complex $\Gamma$ when $\Delta$ is a flag simplicial sphere \cite{CCV} and the compatibility between Chebyshev polynomials and a variant of the $f$-polynomial by work of Hetyei \cite{He2} ($F$-polynomial from Definition \ref{tchebdef}, Proposition \ref{tchebcompat}). 
	
	\begin{thm} (Hetyei, Proposition 3.3 on p. 578 -- 579 of \cite{He2}) \label{TchebFcompat}
		\\
		Given a simplicial complex $A$, let the $F$-polynomial of $A$ be $F(x) \coloneq f \left( \frac{x - 1}{2} \right)$. The Tchebyshev transform of the $F$-polynomial of $A$ is the $F$-polynomial of its Tchebyshev triangulation (ordered repeated edge subdivisions): \[ T(F_A)(x) = F_{T(A)}(x). \] 
	\end{thm}
	
	The main geometric transformation here is repeated (stellar) subdivision of all the edges in some order (Tchebyshev subdivision from Part 2 of Definition \ref{tchebdef}) and a formal generalization of this (Definition \ref{Tchebdivgen}, Remark \ref{tchebdivgencom}). This yields a sort of type A to type B transition (e.g. Coxeter complexes from type A to type B, dually permutohedra of these types from \cite{He3}). Heuristically, this is connected to transitions between ``Boolean structures'' (e.g. face posets of simplicial complexes) and cubical complexes (see Theorem \ref{intchebsubdiv}, Proposition \ref{barcover}, Remark \ref{indexcon}). \\
	
	Combining this with the earlier inverted Chebyshev expansion, we show that the reciprocal polynomial of the gamma polynomial, the $f$-polynomial of the Tchebyshev subdivision $T(\mathcal{D}(\Gamma))$ of an auxiliary complex $\mathcal{D}(\Gamma)$ constructed out of a balanced simplicial complex $\Gamma$ such that $h(\Delta) = f(\Gamma)$ for our flag simplicial sphere $\Delta$, and a cubical complex (mirroring/Danzer/power complex from Defintion \ref{mirdancom}) with local structure given by $T(\mathcal{D}(\Gamma))$ (Corollary \ref{invtchebsub}).

	\begin{thm} \label{invtchebsubintro} (Corollary \ref{invtchebsub}) \\
		Recall that \[ g_\Delta(u) \coloneq h_{ \frac{d}{2} }(\Delta) + 2 \sum_{j = 1}^{ \frac{d}{2} } h_{ \frac{d}{2} - j }(\Delta) T_j \left( \frac{u}{2} \right). \]
		
		We have that \[ \frac{g_\Delta(2w) + f_{ \frac{d}{2} - 1 }(\Gamma) }{2} - \sum_{i = 0}^{ \frac{d}{2} } f_{ \frac{d}{2} - j - 1 }(\Gamma) + 1  = F(T(\mathcal{D}(\Gamma)), w) = \frac{1}{2^{ n(T(\mathcal{D}(\Gamma))) }  } \widetilde{f}(M(T(\mathcal{D}(\Gamma)), w - 1), \] where the last equality follows from Proposition \ref{Fpolytodanz}. \\
	\end{thm}
	
	This ``local structure'' construction is motivated by a concrete construction using flag simplicial spheres $\Delta$ to build $CAT(0)$-complexes with repeated local structure given by $\Delta$ (Proposition 1.2.2 on p. 11 and Proposition 1.2.3 on p. 13 of \cite{DCox}, p. 21 of \cite{DM}). Note that an intermediate step between transformations between $F$-polynomials and $f$-polynomials are connected to face count conversions between these complexes (p. 302 of \cite{BBC}). The complex $\mathcal{D}(\Gamma)$ determining the local structure comes from unused colors of faces in ``signed proper $d$-colorings'' of the 1-skeleton of $\Gamma$ (``signed unused color complex'' from Corollary \ref{invtchebsub} and Definition \ref{sigunuscol}). In particular, the main formula involved (Corollary \ref{invtchebsub}) implies that reciprocal polynomial of the gamma polynomial with the variable multiplied by 2 has nonnegative coefficients (Part 1 of Corollary \ref{simpgamdanz}). \\

	Apart from Tchebyshev subdivisions and conversions to cubical complexes, another connection to cubical complex structures is a cubical analogue of barycentric subdivisions (Definition \ref{cubbary}). The $f$-vector transformation taking a simplicial complex to its cubical barycentric subdivison comes from translations, which naturally occur the positivity properties mentioned above. They also appear as an intermediate step in transformations between $F$-polynomials and $f$-polynomials. We note that our motivation for considering cubical complex structures came from studying the behavior of the expression when $\Delta$ is the barycentric subdivision of a simplicial sphere, in which case we have $\gamma(\Delta) = f(T)$ for some simplicial complex $T$ by work of Nevo--Petersen--Tenner (Theorem 1.1 on p. 1365 and Corollary 5.6 on p. 1376 of \cite{NPT}).

	Putting this together with the framework described in the previous paragraph, we find that there is a close connection between the $f$-polynomials of the cubical complexes above and cubical barycentric subdivisions and those of ``dual complexes'' of $T$ (Part 2 of Corollary \ref{simpgamdanz}). The following gives new positivity properties and formal structural results (where ``formal'' can be dropped in cases such as barycentric subdivisions of simplicial spheres). \\
	
	\begin{cor} \label{simpgamdanzintro} ~\\
		\vspace{-3mm}
		\begin{enumerate}
			\item Suppose that $h(\Delta)$ is the $h$-vector of a flag sphere of odd dimension $d - 1$ (i.e. $d$ is even). If $d \ge 4$, then the polynomial $g_\Delta(2(u + 1)) = (2(u + 2))^{\frac{d}{2}} \gamma_\Delta \left( \frac{1}{2(u + 2)} \right)$ has nonnegative coefficients. After multiplying by a power of 2, we have the $f$-polynomial of the mirroring/Danzer/power complex $M(T(\mathcal{D}(\Gamma)))$ with a nonnegative constant added to it. \\

			\item By the $f_k^\Box(S)$ formula below Definition \ref{cubbary}, substituting the variable $w$ by $u + 1$ sends the $f$-polynomial of a simplicial complex to that of its cubical barycentric subdivision. Consider the $f$-polynomial of the mirroring/Danzer/power complex $M(T(\mathcal{D}(\Gamma)))$. \\
			
			Corollary \ref{invtchebsub} implies that its $f$-polynomial behaves ``formally''  like a double cubical barycentric subdivision associated to some integer tuple. If $\gamma(\Delta) = f(T)$, we can remove the word ``formal'' and take the $f$-polynomial associated to a cell complex after a rescaling as the input.  \\

		\end{enumerate}
		
		\color{black}
	
	\end{cor}

	Recalling that simplicial complexes $\Gamma$ such that $h(\Delta) = f(\Gamma)$ have been a key input in our steps, this is followed by a discussion of the recursive structures of such $\Gamma$ in Section \ref{gamfdec}. In the barycentric subdivision case that led us to study connections between gamma vectors and $f$-vectors, the following a main underlying structure used. \\
	
	\begin{prop} (Nevo--Petersen--Tenner, Proposition 6.2 on p. 1377 of \cite{NPT}) \\
		Let $\Delta$ be a flag sphere. If $\gamma(\Delta)$ is the $f$-vector of a simplicial complex $S$, then $h(\Delta)$ is the $f$-vector of the following simplicial complex:
		
		\[ \Gamma = \{ F \cup G : F \in \mathcal{F}(\gamma(\Delta)), G \in 2^{[d - 2|F|]} \}, \] where $\mathcal{F}(T)$ is the standard compressed complex associated to $T$ taking $\frac{d}{2}$-compressed complexes $\mathcal{F}_k(f(T))$ and letting $k \to \infty$ (see p. 263 of \cite{Zie} and p. 1366 -- 1369 of \cite{NPT}).
		
		Note that the part involving $|F|$ is multiplied by $2$. \\
	\end{prop}
	
	This includes tracking changes under edge subdivisions and studying the vertex decomposable case.  \\

	\begin{thm} (Proposition \ref{edgegenbool}, Corollary \ref{vertdecbool})
		\begin{enumerate}
			\item (Proposition \ref{edgegenbool}) The edge subdivision of a simplicial complex that has a Boolean decomposition has a generalized Boolean decomposition (Definition \ref{booldecomp}).
			
			\item (Part 1 of Corollary \ref{vertdecbool}) Suppose that a simplicial complex $\Delta$ has an associated simplicial complex $\Gamma$ such that $h(\Delta) = f(\Gamma)$ and $\Gamma$ has a Boolean decomposition. Let $\Delta'$ be the (stellar) subdivision of $\Delta$ with respect to an edge $e \in \Delta$. Then, there is a simplicial complex $\Gamma'$ such that $h(\Delta') = f(\Gamma')$ and $\Gamma$ has a generalized Boolean decomposition. If $\lk_\Delta(e)$ also has $h(\lk_\Delta(e)) = f(\Gamma_e)$ for a simplicial complex $\Gamma_e$ with a Boolean decomposition, then $\Gamma *_{\Gamma_e * v} u$ and $\Ast_\Gamma(p) *_{\Gamma_e} u$ also have Boolean decompositions for any ``Boolean'' vertices $u$, $v$, and $p$ of $\Gamma$. \\
			
			\item (Part 2 of Corollary \ref{vertdecbool}) Suppose that $\Delta$ is a simplicial complex satisfying one of the following properties:
			
			\begin{itemize}
				\item $\Delta$ is vertex decomposable and flag.
				
				\item $\Delta$ is a balanced, vertex decomposable, flag simplicial complex.
			\end{itemize}

			Then, these properties are also satisfied by $\Ast_\Delta(v)$ and $\lk_\Delta(v)$ in the respective cases. Also, the existence of Boolean decompositions for simplicial complexes $\Gamma_1$ and $\Gamma_2$ such that $h(\Ast_\Delta(v)) = f(\Gamma_1)$ and $h(\lk_\Delta(v)) = f(\Gamma_2)$ would imply that one exists for $\Gamma \coloneq \Gamma_1 *_{\Gamma_2} u$ for a new vertex $u$ (noting that $h(\Delta) = f(\Gamma)$). \\ 
		\end{enumerate}
	\end{thm}

	Recursive properties common to both settings are listed in Proposition \ref{commrec}). \\

	We start our approach to studying connections between gamma vector and face structures with a review of Tchebyshev transforms/subdivisions and compatibility relations in Section \ref{tchebreview}. This is followed by setting up background information for analyzing related geometry of the simplicial complexes giving the relevant $f$-vectors in Section \ref{background}. A transformation that will occur frequently is taking the poset of intervals partially ordered by inclusion in Section \ref{approachback}. Putting this together, we find a connection between Tchebyshev subdivisions (repeated edge subdivisions) of ``unused coloring complexes'' (Definition \ref{sigunuscol}) and interval posets of $S$ such that $\gamma(\Delta) = f(S)$ (Corollary \ref{simpgamdanz}). Finally, we use generalizations of Tchebyshev subdivisions to cell complexes and the idea behind a proof of a result of Rowlands to study connections between $f$-vectors of flag complexes coming from gamma vector constructions and $CAT(0)$ complexes. Afterwards, we give an overview of Boolean decompositions and gamma positivity followed by an analysis of its behavior under edge subdivisions (Proposition \ref{edgegenbool}) and recursive properties connected to vertex decomposability (Corollary \ref{vertdecbool}).

	\color{black}

	\color{black}

	\section{Geometry of the gamma vector as a whole via Tchebyshev transformations and subdivisions \\} \label{chebexpcellgeo}

	While $h$-vectors are typically perceived as more ``geometric'' adapting information from toric varieties, we find that the combinatorial methods of thinking about $f$-vectors and the inverted Chebyshev expansion from Theorem \ref{gamchebinv} gives interesting structural information from a different perspective. In particular, we obtain new positivity properties and underlying cell complex structure describing the gamma vector as a whole. The main formula obtained is in Corollary \ref{invtchebsub} and implications including positivity and structural properties (e.g. repeated edge subdivisions, cubical analogues of barycentric subdivisions) are given in Corollary \ref{simpgamdanz}. This is enabled by the compatibility of $f$-vector data with repeated edge subdivisions encoded by linear transformations induced by Chebyshev polynomials studied in earlier work of Hetyei. We also have a deeper connection with cubical complexes beyond the top component via transformations from faces of flag simplicial spheres $\Delta$ to $CAT(0)$-complexes with repeated local structure given by $\Delta$ (Remark \ref{locflagspherecat}). Along the way, we find a sort of conversion from type A to type B data (keeping Coxeter complexes in mind) that seems analogous to switching between face posets of simplicial and cubical complexes (e.g. multiplication by 2 involved). Following initial definitions from Section \ref{tchebreview} and Section \ref{background}, the case of barycentric subdivisions of flag spheres is used to introduce and motivate transitions between ``Boolean structures'' (e.g. face poset of a simplicial complex) and cubical complexes and cubical barycentric subdivisions involved in Section \ref{approachback}. We note that the relationship between $f$-polynomials and $F$-polynomials is connected to both the local structure discussion above and cubical barycentric subdivisions giving cubical complexes. \\
	
	\color{black}
	
	\subsection{Background on Tchebyshev transforms/subdivisions and compatibility} \label{tchebreview}

	Before we study structural properties of cell complexes connected to Tchebyshev transformations, we first review some key definitions and properties mentioned in the introduction. \\
	
	\color{black}

	\begin{defn} (Hetyei, Definition 1.1 on p. 571 and Definition 2.1 on p. 574 of \cite{He2}, Definition 2 on p. 921 of \cite{He3}) \label{tchebdef} \\
		\begin{enumerate}
			\item Given a polynomial $F(x) = a_n x^n + \ldots + a_1 x + a_0 \in \mathbb{R}[x]$, the \textbf{Tchebyshev transforms} $T(F)(x)$ and $U(F)(x)$ \textbf{of the first and second kind } of $F(x)$ are given by 
			\[ T(F)(x) \coloneq a_n T_n(x) + \ldots + a_1 T_1(x) + a_0 \]
			and \[ U(F)(x) \coloneq a_n U_{n - 1}(x) + \ldots + a_1 U_0(x), \]
			
			where $T_m(x)$ is Tchebyshev polynomial of the first kind, determined by \[ T_m(\cos(\alpha)) = \cos(m \cdot \alpha) \] and $U_m(x)$ is the Tchebyshev polynomial of the second kind for all $m \ge 0$. \\
			
			\item  We define a \textbf{Tchebyshev triangulation} $T(\Delta)$ of a simplicial complex $\Delta$ having $m = f_1(\Delta)$ as follows. Number the edges $e_1, \ldots, e_m$ in some order. We subdivide the edge $e_i$ into a path of length 2 by adding the midpoint $w_i$ and we also subdivide all the faces containing $e_i$ by performing a stellar subdivision. Note that the $f$-vector from any initial ordering yields the same $f$-vector (and thus the same $F$-vector) (Theorem 3 on p. 922 of \cite{He3}). \\
		\end{enumerate}
		
	\end{defn}

	The Tchebyshev subdivision is compatible with a modification of the $f$-polynomial of a simplicial complex. \\

	\begin{prop} (Hetyei, Hetyei--Nevo from Proposition 3.3 on p. 578 -- 579 and Proposition 4.4 on p. 580 -- 581 of \cite{He2}, Proposition 5.1 on p. 99 of \cite{HN})  \label{tchebcompat} \\
		Given a simplicial complex $S$, let $F_S(x) \coloneq f_S \left( \frac{x - 1}{2} \right)$. Here, $f_\Delta(t)$ is the usual $f$-polynomial of a simplicial complex $\Delta$.
		
		\begin{enumerate}
			\item The Tchebyshev transform of the $F$-polynomial of a simplicial complex is the $F$-polynomial of its Tchebyshev triangulation: \[ T(F_\Delta)(x) = F_{T(\Delta)}(x). \]
			
			\item  The Tchebyshev transform of the second kind of the $F$-polynomial of a simplicial complex is half of the $F$-polynomial of its Tchebyshev triangulation: \[ U(F_\Delta)(x) = \frac{1}{2} F_{U(\Delta)}(x). \] Note that a Tchebyshev subivision is also the same transformation taking the type $A$ Coxeter complex to the type $B$ Coxeter complex \cite{He3}. Dually, one can consider faces of type A and type B permutohedra (p. 919 of \cite{He3}). \\
		\end{enumerate}
	\end{prop}

	\subsection{Cell complex structures induced by the $f$-vector perspective} \label{tchebcell}

	\subsubsection{Background related to cell and cubical complexes} \label{background}

	We give some background on objects that we will be studying. Recall that the $w = 2 \alpha + 1$ substitution was motivated by trying to use $f_S(x) = F_S(2x + 1)$ for a simplicial complex $S$ to show that the input for the Tchebyshev transformation (for polynomials) in $T(P_\Delta(w)) = g_\Delta(2w) = (2w + 2)^{ \frac{d}{2} } \gamma_\Delta \left( \frac{1}{2w + 2} \right)$ is the $F$-vector of some simplicial complex. However, the indices from Theorem \ref{danzinput} are flipped in that we have the index $\frac{d}{2} - k - 1$  term for the coefficient of $\alpha^k$ instead of something of index $k - 1$. \\

	A cubical complex that is closely connected to $F$-polynomials we are considering the mirroring/Danzer/power complex $M\Delta$ of a cell complex $\Delta$, which is a cubical complex whose vertices have links isomorphic to $\Delta$. Before listing the definition of this cubical complex, we define what we mean by a cell complex in our setting.

	\begin{defn} \label{celldef} ~\\
		\vspace{-2mm}
		\begin{enumerate}
			\item By a \textbf{cell complex}, we mean a collection of subsets of the vertex set satisfying the formal properties of a polyhedral complex (Definition 5.1 on p. 127 of \cite{Zie}) or the (polar) dual of one (analogous to polar duals of polytopes, defined in Part 2) or a modification of one by one of the following transformations:

			\begin{itemize}
				\item A direct generalization of \textbf{Tchebyshev subdivisions} from Definition 2.1 on p. 573 of \cite{He2} (see Definition \ref{tchebdef} and Remark \ref{tchebdivgencom} for intuitive idea and comments on generalizations)
				
				\item Suppose that we have a proper coloring of the 1-skeleton of the given (dual of) a polyhedral complex and that each vertex is assigned a $\pm$ sign. A \textbf{``signed unused color complex''} keeps track of $\pm$ signs of vertices on unused colors of a face of the starting complex (Definition \ref{sigunuscol}). Note that the number of colors used will be fixed as part of the initial conditions. \\
			\end{itemize}
			
			\color{black} 
			\item  We will take the \textbf{dual} $A^*$ of a cell complex $A$ of dimension $D - 1$ to mean the cell complex where a face of dimension $j$ corresponds to the collection of facets of $A$ that contain a particular $(D - j - 1)$-dimensional face of $A$. This is an analogue of the polar dual of a polytope. \\
		\end{enumerate} 
	\end{defn}

	\begin{rem}
		The definitions above generalize directly to multicomplexes studied in \cite{BFS}. We mention this since the identities of the form $h(\Delta) = f(\Gamma)$ used in our analysis (e.g. for flag Cohen--Macaulay simplicial complexes from \cite{CCV}) also exist for multiplicomplexes.
	\end{rem}
	
	\color{black}

	We now give the definition of the cubical complex mentioned above. 
	
	\color{black} 
	\begin{defn} (a small generalization of the definition in p. 299 of \cite{BBC}, p. 20 -- 21 of \cite{DM}, and p. 10 -- 11 of \cite{DCox}) \label{mirdancom} \\
		Let $T$ be a cell complex. We can think of $T$ (including the empty set) as a partially ordered set (the corresponding simplicial poset) where each face is a 0-1 vector with a 0 in the $i^{\text{th}}$ place if and only if the vertex $i$ belongs to the face. Thus $T$ is partially ordered by $0 > 1$ extended componentwise. \\
		
		Then we construct a partially ordered set $MT$ (called the $\textbf{mirroring}$, $\textbf{Danzer complex}$, or \textbf{power complex $2^T$} of $T$) as follows: \[ MT \coloneq \{ (a_1, \ldots, a_n) : (|a_1|, \ldots, |a_n|) \in T \} \subset I^n, \] partially ordered by $0 > 1$ and $0 > -1$ extended componentwise. Here $I$ is the poset $\{ 0, 1, -1 \}$ with ordering $1 < 0$ and $-1 < 0$. Note that $\dim MT = \dim T + 1$. \\
	\end{defn}

	\begin{rem} \textbf{(Comments on notation) \\} \label{mirdanzgeom}
		The way we will think about the notation above is that placing $1$ or $-1$ in the $i^{\text{th}}$ coordinate refers to intersecting with the hyperplane $x_i = 1$ or $x_i = -1$ defining a facet of the $n$-cube $[-1, 1]^n$ and that the coordinates labeled with a ``0'' are ``free'' in the sense that they are used to span a face of the cubical complex $MT$. In Rowlands' thesis (see p. 4 of \cite{Row}), the ``free'' coordinates are labeled with a ``*'', which may better convey this. As a consequence, the faces of $MT$ with $0$'s in the same coordinates (i.e. coming from the same face of $T$) are parallel to each other. \\
	\end{rem}

	\begin{rem} \textbf{(Flag simplicial spheres to $CAT(0)$ complexes) \\} \label{locflagspherecat}
		The mirroring/Danzer/power complex gives rise to a construction taking flag simplicial spheres to $CAT(0)$ complexes with repeated local structure given by $\Delta$ This ``local structure'' construction is motivated by a concrete construction using flag simplicial spheres $\Delta$ to build $CAT(0)$-complexes with repeated local structure given by $\Delta$ (Proposition 1.2.2 on p. 11 and Proposition 1.2.3 on p. 13 of \cite{DCox}, p. 21 of \cite{DM}). \\
	\end{rem}
	
	\pagebreak 
	
	\color{black}
	\begin{exmp} \textbf{(Intersection patterns of faces on $M \Delta$)} \label{danzintpat} \\
		\vspace{-3mm}
		\begin{enumerate}
			\item Let $[n]$ be the vertex set of $\Delta$. A collection of faces $G_1, \ldots, G_m \in M\Delta$ has a nonempty intersection if and only if the following condition holds: For all $1 \le i \le n$, $(G_k)_i, (G_\ell)_i \ne 0 \Longrightarrow (G_k)_i = (G_\ell)_i$. \\
			
			\item Suppose that the faces $G_1, \ldots, G_m \in MT$ have a nonempty intersection. For each $1 \le i \le m$, let $F_i = \{ (|a_1|, \ldots, |a_n|) : (a_1, \ldots, a_n) \in G_i \}$ (i.e. subset labeled by 0's corresponds to $F_i$). Then, we have that $G_1 \cap \cdots \cap G_m$ has 0's in the same slots as $F_1 \cap \cdots \cap F_m \in \Delta$ and the ``common values'' from Part 1 for coordinates corresponding to elements of the complement $[n] \setminus (F_1 \cap \cdots \cap F_m)$. In particular, this means that $F_1 \cap \cdots \cap F_m \in M \Delta$ since the subset of $[n]$ labeled with 0's is the face $F_1 \cap \cdots \cap F_m \in \Delta$ and the rest of the coordinates are labeled with $\pm 1$. \\
			
			\item In inclusions of faces of $M\Delta$, the changes in coordinates can only come from additional 0's being placed that are induced by inclusions of faces of $\Delta$. Given a face $G \in M\Delta$. This implies the following:

			\begin{itemize}
				\item  $\St_{M\Delta}(G) \cong \St_\Delta(|G|)$, where $|G|$ denotes the face of $\Delta$ obtained after taking the absolute value of each coordinate (i.e. setting the $-1$ coordinates of $G$ equal to 1).
				
				\item The nonzero coordinates of elements of $\St_{M\Delta}(G)$ are determined by those of $G$ (which are indexed by $[n] \setminus |G|$). \\
			\end{itemize}
			
		\end{enumerate}
	\end{exmp}

	Returning to the original question of the interpretation of $P_\Delta(2\alpha + 1)$, we use the following result connecting faces of a cell complex with its mirroring/Danzer/power complex. \\

	\begin{defn} (\textbf{$\widetilde{f}$-polynomial of a poset}, p. 307 of \cite{BBC}) \label{posetfpoly} \\
		Given a poset $P$ such that $\rk(P) = D$, we denote by $\widetilde{f}_i \coloneq \widetilde{f}_i(P)$ the number of elements of rank $i$ and define the polynomial \[ \widetilde{f}(P, t) \coloneq \sum_{i = 0}^D \widetilde{f}_i t^i. \]
	\end{defn}

	\begin{rem} \textbf{(Comments on indexing conventions for simplicial vs. cubical complexes) \\} \label{indexcon}
		We write $\widetilde{f}_i$ in a change from the notation in the original work of Babson--Billera--Chan \cite{BBC} since there may be a difference from the usual definition of the $f$-polynomial of a simplicial complex due to the indexing conventions used there for face posets of simplicial complexes and cubical complexes (see discussion on p. 298 of \cite{BBC}). To be more precise, the empty set is included in the face poset of a simplicial complex while it is \emph{not} included in the face poset of a cubical complex. This is to preserve the property of the order ideals of these face posets being Boolean algebras and products of copies of the face poset of an interval respectively. Another source where the poset $\widetilde{f}$-polynomial is used is on p. 13 of work of Rowlands \cite{Row2} on connections between $CAT(0)$ cubical complexes and flag simplicial complexes (where $f$ is taken to mean $\widetilde{f}$ in our setting). \\

		As a result, we have that $\widetilde{f}_i(S) = f_{i - 1}(S)$ for a simplicial complex $S$ and $\widetilde{f}_j(K) = f_j(K)$ for a cubical complex $K$ when we take $f_\ell(A)$ to mean the number of $\ell$-dimensional faces of a polyhedral complex $A$ as usual. We take a look at what this means for the comparison between \[ \widetilde{f}(P, t) \coloneq \sum_{i = 0}^D \widetilde{f}_i(P) t^i \] and the ``usual'' $f$-polynomial \[ f_A(t) \coloneq \sum_{i = 0}^D f_{i - 1}(A) t^i. \]
		
		Since $\widetilde{f}_i(S) = f_{i - 1}(S)$ for a simplicial complex $S$, we actually have that $\widetilde{f}(S, t) = f_S(t)$. However, we need to make a modification for a cubical complex $K$ since $\widetilde{f}_i(K) = f_i(K)$ for a cubical complex $K$. In that case, we have that
		
		\begin{align*}
			\widetilde{f}(K, t) &= \sum_{i = 0}^D \widetilde{f}_i(K) t^i \\
			&= \sum_{i = 0}^D f_i(K) t^i \\
			&= f_0(K) + f_1(K) t + f_2(K) t^2 + \ldots + f_D(K) t^D \\
			\Longrightarrow t \widetilde{f}(K, t) &= f_0(K) t + f_1(K) t^2 + f_2(K) t^3 + \ldots + f_D(K) t^{D + 1} \\
			&= f_D(K) - 1 \\
			\Longrightarrow f_K(t) &= t \widetilde{f}(K, t) + 1. 
		\end{align*}
		
	\end{rem}
	
	Using this notation, there is a clean way to connect the $\widetilde{f}$-polynomials of a cell complex $T$ and its mirroring/Danzer complex $MT$. \\

	\begin{prop} (extension of p. 302 of \cite{BBC}) \label{posetfdanzer} \\
		Suppose that we include the empty set in the face poset of a cell complex $T$ and that it is \emph{not} included in the face poset of $MT$. If every $(k - 1)$-dimensional cell of $T$ has $k$ vertices, we have that \[ \widetilde{f}(MT, t) = 2^{n(T)} \widetilde{f} \left( T, \frac{t}{2} \right), \] the notation of Definition \ref{posetfpoly} where $n(T)$ denotes the number of vertices in $T$. \\
		
		In general, we have that \[ \widetilde{f}(MT, t) = 2^{n(T)} \widetilde{S} \left(T, \frac{t}{2} \right), \] where $\widetilde{S}(T, x) = \sum_{i = 0}^{n(T)} \widetilde{s}_i x^i$ with $\widetilde{s}_i \coloneq \# \{ F \in T : |F| = i \}$. 
	\end{prop}

	\subsubsection{Colorings, Tchebyshev subdivisions, and intervals} \label{approachback}

	The main aim of this section is to connect inverted Chebyshev expansions associated to gamma polynomials of flag simplicial spheres to cell complex structures. Our main tools are the compatibility between Tchebyshev subdivisions and $F$-polynomials from Section \ref{tchebtypeAtoB} and the $h$-vector itself being equal to the $f$-vector of an auxiliary simplicial complex in this setting (p. 470 of \cite{CCV}, Theorem 1 on p. 23 and Section 6.4 on p. 33 of \cite{BFS}). One can think about the type $A$ to type $B$ transitions in terms of converting information about ``Boolean structures'' (e.g. face posets of simplicial complexes) into those about cubical complexes via repeated edge subdivisions. We start by considering flag simplicial spheres $\Delta$ where $\gamma(\Delta) = f(T)$ (e.g. barycentric subdivisions of simplicial spheres \cite{NPT}). This will introduce transformations of cubical complexes compatible with those of the polynomials we study (e.g. translations) that describe the simplicial to cubical complex transitions in a more concrete way. \\

	Note that having $h(\Delta) = f(\Gamma)$ is a necessary condition in order for this occur (Proposition 6.2 on p. 1377 of \cite{NPT}). When $\Delta$ is a flag sphere, this equality holds and $\Gamma$ can be taken to be a balanced simplicial complex. Combining this with the compatibility between Tchebyshev transformations and $F$-polynomials, we obtain our main formula connecting the gamma polynomial to $F$-polynomials of an auxiliary cell complex coming from signed colorings of $\Gamma$ (Corollary \ref{invtchebsub}). The $F$-polynomial is the translation by -1 of the generating function for the number of faces of a cubical complex whose local structure is this determined by this auxiliary cell complex. The motivation comes a certain construction that converts flag simplicial spheres $\Delta$ into $CAT(0)$-complexes whose local structure is given by $\Delta$ (Proposition 1.2.2 on p. 11 and Proposition 1.2.3 on p. 13 of \cite{DCox}, p. 21 of \cite{DM}, p. 302 of \cite{BBC}). \\
	
	Finally, we use this formula to show new positivity properties involving the reciprocal polynomial of the gamma polynomial that hold for arbitrary flag spheres and that the $f$-polynomial of this cubical complex ``formally'' behaves like a cubical version of a barycentric subdivision (where the word ``formally'' can be removed when $\gamma(\Delta) = f(T)$) (Corollary \ref{simpgamdanz}). Note that the this comes from a translation, which is also related to transformations between $F$-polynomials and $f$-polynomials. This is followed by some further context involving earlier work studying connections between $CAT(0)$ cubical complexes and flag simplicial complexes. \\
	
	\color{black}
	Now suppose that $\gamma(\Delta) = f(T)$ for some simplicial complex $T$ of dimension $\le \frac{d}{2} - 1$. By Theorem \ref{gamchebinv}, we have that \[ g_\Delta(2w) = (2w + 2)^{ \frac{d}{2} } \gamma_\Delta \left( \frac{1}{2w + 2} \right) = (2w + 2)^{ \frac{d}{2} } f_T \left( \frac{1}{2w + 2} \right). \]

	\begin{align*}
		g_\Delta(2w) &= \sum_{i = 0}^{ \frac{d}{2} } f_{i - 1}(T)(2w + 2)^{ \frac{d}{2} - i} \\
		&= \sum_{i = 0}^{ \frac{d}{2} } f_{i - 1}(T) 2^{ \frac{d}{2} - i } (w + 1)^{ \frac{d}{2} - i } \\
		&= \sum_{i = 0}^{ \frac{d}{2} } \widetilde{f}_i(T) 2^{ \frac{d}{2} - i } (w + 1)^{ \frac{d}{2} - i } \\
		&= \sum_{j = 0}^{ \frac{d}{2} } \widetilde{f}_{ \frac{d}{2} - j }(T) 2^j (w + 1)^j \\
		\Longrightarrow [w^k] g_\Delta(2w) &= \sum_{j = k}^{ \frac{d}{2} } \widetilde{f}_{ \frac{d}{2} - j }(T) 2^j \binom{j}{k} \\
		&= \sum_{j = k}^{ \frac{d}{2} } \widetilde{f}_j(T^*) 2^j \binom{j}{k}
	\end{align*}
	
	since $\widetilde{f}_i(S) = f_{i - 1}(S)$ for a simplicial complex $S$ by Remark \ref{indexcon}. \\

	We note that this has some resemblance to first cubical barycentric subdivision of a simplicial complex mentioned defined by Hetyei (see \cite{Hetinv}, which is defined below). We will also give a generalization of Babson--Billera--Chan \cite{BBC} which puts it in context of well-known subdivisions. \\
	
	\begin{defn} (Hetyei, D\'efinition 6 on p. 45 and Remarque 5 on p. 45 -- 46 of \cite{Hetinv}) \label{cubbary} \\
		The \textbf{first cubical barycentric subdivision} of a simplicial complex $S$ of dimension $D - 1$ is the cubical complex $\mathcal{C}(S)$ defined as follows:
		
		\begin{enumerate}
			\item The set of vertices of $\mathcal{C}(S)$ is the set of nonempty faces of $S$. 
			
			\item  The faces of $\mathcal{C}(S)$ are the intervals of the face poset of $S$, that is all sets of the form \[ [\sigma, \tau] \coloneq \{ \lambda \in S : \sigma \subset \lambda \subset \tau \}. \]
		\end{enumerate}
		
		We can express the $f$-vector $(f_{-1}^\Box(S), \ldots, f_{D - 1}^\Box(S))$ of the complex $\mathcal{C}(S)$ in terms of the $f$-vector $(f_{-1}(S), \ldots, f_{d - 1}(S))$ of $S$ as follows: \[ f_k^\Box(S) = \sum_{j = k}^{d - 1} \binom{j + 1}{k} f_j(S) \]
		
	\end{defn}
	
	\begin{rem} \textbf{(Comparison with the formula for $[w^k] g(2w)$ given if $\gamma(\Delta) = f(T)$) \\}
		\vspace{-3mm}
		\begin{enumerate}
			\item Comparing this to the formula for $[w^k] g(2w)$ obtained when $\Gamma(\Delta) = f(T)$, we see that it uses $\binom{j}{k}$ instead of $\binom{j + 1}{k}$ and has an input that might not be a simplicial complex. Note that the $k$-dimensional faces of $MA$ come from $(k - 1)$-dimensional faces of $A$ (see Definition \ref{mirdancom}). 
			
			\item The cubical barycentric subdivision assumes that both endpoints of the interval considered are contained in $S$. If we drop the condition that $S$ is a simplicial complex, we're simply considering endpoints $[\sigma, \tau]$ where the upper endpoint is a $j$-dimensional face of $S$ and $\tau$ is some subset of $\sigma$ (which may or may not be a face of $S$). 
		\end{enumerate}
	\end{rem}
	
	A generalization by Babson--Billera--Chan \cite{BBC} is defined below.
	
	\begin{defn} (Babson--Billera--Chan, Section 2.3 on p. 303 -- 304 of \cite{BBC}) \\
		Given a poset $P$, its \textbf{barycentric cover} $KP$ is the poset with elements the order relations of $P$, partially ordered by inclusion (i.e. $(u \le v) \le (x \le y)$ (see Definition \ref{intposetdef}) if and only if $x \le u \le v \le y$). Note that $KP = K^{\text{op}}$ (see p. 298 of \cite{BBC}) and $K(P \times Q) = K(P) \times K(Q)$. 
	\end{defn}
	
	Here are some properties of the barycentric cover $KP$.
	
	\begin{prop} (Babson--Billera--Chan, Proposition 2.2 on p. 303 of \cite{BBC}) \label{barcover} \\
		If $P$ has Boolean intervals, then $KP$ is a cubical poset. This includes simplicial and cubical posets (more generally, face posets of polyhedral complexes with simple, nonempty cells) as well as their duals. Further, if $\widehat{P}$ ($P$ with $\widehat{0}$ and $\widehat{1}$ adjoined) is a lattice, then so is $\widehat{KP}$. \\
	\end{prop}
	
	This explains the ``cubical'' part and the induced triangulation of cubes in $KP$ induced by the barycentric subdivision give the cubical barycentri subdivsion name (see top of p. 304 of \cite{BBC}). More information on structural properties of $KP$ are listed below. \\
	
	\begin{thm} (Babson--Billera--Chan, Theorem 2.3 and Corollary 2.4 on p. 303, p. 298 of \cite{BBC})
		\begin{enumerate}
			\item Given a poset $Q$, let $|Q|$ be the (simplicial) complex of chains in $P$ (i.e. the order complex). Then, $|KP|$ gives a polyhedral subdivision of $|P|$ by the map taking the point $(x \le x)$ to itself and the point $(x < y)$ in $|KP|$ to the midpoint of the edge $x < y$ in $|P|$, and extending linearly over every closed simplex in $|KP|$. \\
			
			\item If $P$ is the face poset of a polyhedral complex, then $KP$ is the face poset of a polyhedral subdivision of $P$, lying between $P$ and $|P|$ in the refinement order. 
		\end{enumerate}
	\end{thm}

	These interval constructions are closely connected to Tchebyshev subdivisions by recent work of Hetyei \cite{He3}.

	\begin{defn} (Definition 8 on p. 924 of \cite{He3}) \label{intposetdef} \\
		An \textbf{interval} $[u, v]$ in a partially ordered set $P$ is the set of all elements $w \in P$ satisfying $u \le v \le w$. For a finite partially ordered set $P$, we define the \textbf{poset $I(P)$ of the nonempty intervals of $P$} as the set of all intervals $[u, v] \subset P$, ordered by inclusion.  
	\end{defn}
	
	The connection to Tchebyshev subdivisions comes from taking order complexes.
	
	\begin{thm} (Walker--Joji\'c--Hetyei, Theorem 9 on p. 925 of \cite{He3}) \label{intchebsubdiv} \\
		For any finite partially ordered set $P$, the order complex $\Delta(I(P))$ of its poset of nonempty intervals is isomorphic to a Tchebyshev triangulation of $\Delta(P)$ as follows: \\
		
		For each $u \in P$, we identify the vertex $[u, u] \in \Delta(I(P))$ with the vertex $u \in \Delta(P)$ and for each nonsingleton interval $[u, v] \in I(P)$, we identify the vertex $[u, v] \in \Delta(I(P))$ with the midpoint of the edge $\{ [u, u], [v, v] \}$. We number the midpoints $[u_1, v_1], [u_2, v_2], \ldots$ in such a way that $i < j$ holds whenever the interval $[u_i, v_i]$ contains the interval $[u_j, v_j]$. \\
	\end{thm}

	We have a similar situation for intervals of graded posets.
	
	\begin{defn} (Definition 10 on p. 926 of \cite{He3}) \label{gradintposetdef} \\
		For a graded poset $P$ we define its \textbf{poset of intervals $\widehat{I}(P)$} as the poset of all intervals of $P$, \textbf{including the empty set}, ordered by inclusion.
	\end{defn}
	
	\begin{prop} (Hetyei, Proposition 12 on p. 927 of \cite{He3}) \\
		Let $P$ be a graded poset and $\widehat{I}(P)$ be its poset of intervals. Then the order complex $\Delta(\widehat{I}(P) -  \{ \emptyset, [\widehat{0}, \widehat{1}] \} )$ is a Tchebyshev triangulation of the suspension of $\Delta(P - \{ \widehat{0}, \widehat{1} \} )$. 
	\end{prop}

	\begin{rem} \textbf{(Structure of $g_\Delta(2w)$ expression after substituting in $f(T)$ for some simplicial complex $T$) \\} \label{fmatchstrat}
		We take a page from p. 13 of Rowlands' paper \cite{Row2}, which shows that $\widetilde{f}(\Box_P, t) = \widetilde{f}(\Delta_P, t + 1)$. Note that any flag complex is the crossing complex $\Delta_P$ of a $PIP$ $P$ (Lemma 3.3 on p. 11 of \cite{Row2}) and any rooted $CAT(0)$ cubical complex is of the form $\Box_P$ for some $P$ (Theorem 2.9 on p. 9 of \cite{Row2}). The proof uses a bijection between dimension $i$ elements of $\Box_P$ and pairs $(A, M)$, where $A$ is the set of consistent antichains of $P$ and $M \subset A$ is a subset with $|M| = i$. The identity follows after making a change in the order of summation. \\

		Another way to say this is that \[ [w^k] \widetilde{f} \left( T(A), \frac{w - 1}{2} \right)  = \sum_{j = k}^{ \frac{d}{2} } \widetilde{f}_j(T^*) 2^j \binom{j}{k}, \] where $A$ is cell complex related to unused colors of faces in a balanced simplicial complex with an assigned coloring. Note that the last expression makes sense since a $(k - 1)$-face of $T(A)$ has $k$ vertices (see Definition \ref{Tchebdivgen}). \\

	\end{rem}
	
	\color{black}

	We can apply the setup above to give an interpretation of $P_\Delta(2 \alpha + 1)$, where \[ P_\Delta(u) \coloneq h_{ \frac{d}{2} }(\Delta) + 2 \sum_{j = 1}^{ \frac{d}{2} } h_{ \frac{d}{2} - j }(\Delta) u^j. \] The motivation is to show that $P_\Delta(u)$ is the $F$-polynomial of some cell complex $A$. Since $f_S(x) = F_S(2x + 1)$, this is equivalent to showing that $P_\Delta(2 \alpha + 1)$ is the $f$-polynomial of some cell complex $A$. As it turns out, the Tchebyshev subdivisions studied by Hetyei \cite{He2} can be generalized to arbitrary cell complexes (in the sense of Definition \ref{celldef}). This implies that $g(2w) = F_{T(A)}(w)$. Note that $g(2w) = (2w + 2)^{ \frac{d}{2} } \gamma \left( \frac{1}{2w + 2} \right)$. \\

	Before moving on, we first record a construction that is used to express $P$ in terms of a modification of the $f$-vector of some cell complex. \\

	\begin{defn} \textbf{(Signed unused color complex) \\} \label{sigunuscol}
		Let $\Gamma$ be a $(D - 1)$-dimensional balanced simplicial complex. Fix a $D$-coloring of the $1$-skeleton of $\Gamma$. We define the \textbf{signed unused color complex} $\mathcal{D}(\Gamma)$ of $\Gamma$ to be the cell complex below. 
		
		\begin{itemize}
			\item The vertices are of the form $\{ q \} \subset \{ q \} \subset C_F$ or $\emptyset \subset \{ q \} \subset C_F$, where $C_F \subset [D]$ is the set of colors unused by $F$. In the first case, we will denote a vertex by $v^+$ and will write $v^-$ in the second case. Note that $C_\emptyset = [D]$ since none of the colors are used. 
			
			\item The $(k - 1)$-faces of the form $B \subset Q \subset C_F$, where $C_F$ is the set of colors in $[D]$ \emph{not} used by $F$, $Q \subset C_F$ is a $k$-element subset, and $B \subset Q$ is an arbitrary subset. This is like considering ``signed'' $k$-element subsets of $C_F$. Note that $C_F \ne \emptyset$ since $Q$ is nonempty for the cases we will consider. 
			
			\item We have an inclusion of faces $(B_1 \subset Q_1 \subset C_{F_1}) \le (B_2 \subset Q_2 \subset C_{F_2})$ if and only if $F_1 \supset F_2$, $B_1 \subset B_2$, and $Q_1 \setminus B_1 \setminus Q_2 \setminus B_2$. \\
			
		\end{itemize}
		
	\end{defn}
	
	Next, we describe when a collection of vertices of $\mathcal{D}(\Gamma)$ forms a face of $\mathcal{D}(\Gamma)$.
	
	\begin{prop} \label{sigunusvertunion}
		Let $\Gamma$ be a balanced simplicial complex of dimension $D - 1$. A collection of vertices $v_1^{ \varepsilon_1 }, \ldots, v_k^{\varepsilon_k} \in V(\mathcal{D}(\Gamma))$ with each $v_i$ in a triple with a face $F_i \in \Gamma$ and sign $\varepsilon_i = \pm 1$ forms a face of $\mathcal{D}(\Gamma)$ using the union of set of unused colors $C_i$ of the $F_i$ if and only if one of the following conditions (in the notation of Definition \ref{sigunuscol}) holds:
		
		\begin{enumerate}
			\item $F_1 \cap \cdots \cap F_k \ne \emptyset$.
			
			\item $F_1 \cap \cdots \cap F_k = \emptyset$ and $C_{i_1} \cup \cdots \cup C_{i_k} = [D]$.
		\end{enumerate}
		
		Note that each $C_i$ is determined by $F_i \in \Gamma$. \\
		
	\end{prop}
	
	\begin{rem}
		The second condition is nontrivial. For example, suppose that $\Gamma$ is the boundary of the octahedron/2-cross polytope (meaning $D = 3$). This is a balanced simplicial complex. Color the vertices $\pm e_i$ with $i$ for $1 \le i \le 3$ by $i$ to get a 3-coloring of its 1-skeleton. Consider the vertex $F_1 = \{ -e_2 \}$ and the edge $F_2 = \{ e_2, e_3 \}$. Then, we have that $C_{F_1} = \{ 1, 3 \}$ and $C_{F_2} = \{ 1 \}$. Since $F_1 \cap F_2 = \emptyset$ and $C_{F_1} \cup C_{F_2} = \{ 1, 3 \}$, the union of any two vertices of $\mathcal{D}(\Gamma)$ associated to these faces cannot form a face of $\mathcal{D}(\Gamma)$. \\
	\end{rem}
	
	\begin{proof}
		Each vertex is of the form $B_i \subset \{ q_i \} \subset C_{F_i}$, where $C_{F_i}$ is the set of unused colors of $F_i \in \Gamma$, $q_i \in C_{F_i}$, and $B_i = \{ q_i \}$ or $\emptyset$. Using the inclusion relation in Definition \ref{sigunuscol}, we have that the union $\{ v_1^{ \varepsilon_1 }, \ldots, v_k^{\varepsilon_k} \}$ is given by $B_1 \cup \cdots \cup B_k \subset \{ q_1, \ldots, q_k \} \subset (C_1 \cup \cdots \cup C_k)_{F_1 \cap \cdots \cap F_k}$. If $F_1 \cap \cdots \cap F_k \ne \emptyset$, this forms a face of $\mathcal{D}(\Gamma)$ since $(P \cap Q)^c = P^c \cup Q^c$ for $P, Q \subset [D]$. Now suppose that $F_1 \cap \cdots \cap F_k = \emptyset$. As mentioned in Definition \ref{sigunuscol}, we have that $C_\emptyset = [D]$. This is only possible if $C_{F_1} \cup \cdots \cup C_{F_k} = [D]$. \\
	\end{proof}

	We can use this cell complex to connect $P_\Delta(u)$ to the $f$-vector of some cell complex. \\
	
	\color{black}

	\color{black} 
	\begin{thm} \label{danzinput}
		Let $\Delta$ be a flag simplicial homology sphere of even dimension $d$ and \[ P_\Delta(u) \coloneq h_{ \frac{d}{2} }(\Delta) + 2 \sum_{j = 1}^{ \frac{d}{2} } h_{ \frac{d}{2} - j }(\Delta) u^j, \] which is $g(2u)$ with $T_k(u)$ replaced by $u^k$ (i.e. $g(2u) = T(P(u))$ and $P(u) = T^{-1} g(2u)$). \\

		Given $k \ge 1$, there is a $\left( \frac{d}{2} - 1 \right)$-dimensional balanced simplicial complex $\Gamma$ such that

		\[ [\alpha^k] \frac{ P_\Delta(2 \alpha + 1) + f_{ \frac{d}{2} - 1 }(\Gamma) }{2} = f_{k - 1} (\mathcal{D}(\Gamma))  \] for $k \ge 1$ and
		
		\color{black} 
		\[ \frac{ P_\Delta(2 \alpha + 1) + f_{ \frac{d}{2} - 1 }(\Gamma) }{2} - \sum_{j = 0}^{ \frac{d}{2} } f_ { \frac{d}{2} - j - 1 } (\Gamma) + 1 = f_{\mathcal{D}(\Gamma)}(\alpha). \]
		
		Substituting in $\alpha = \frac{\beta - 1}{2}$, we have that
		
		\[ \frac{ P_\Delta(\beta) + f_{ \frac{d}{2} - 1 }(\Gamma) }{2} - \sum_{j = 0}^{ \frac{d}{2} } f_ { \frac{d}{2} - j - 1 } (\Gamma) + 1 = F_{\mathcal{D}(\Gamma)}(\beta), \]
		
		where \[ F(u) \coloneq \sum_{i = 0}^D f_{i - 1}(A) \left( \frac{u - 1}{2} \right)^i \] for a cell complex $A$ of dimension $D - 1$.

	\end{thm}
	
	\color{black}
	
	\begin{proof}
		Since $\Delta$ is a flag simplicial homology sphere, it is also a flag Cohen--Macaulay simplicial complex (e.g. see p. 21 of \cite{Athsom}). Then, a combining a result of Caviglia--Constantinescu--Varbaro (Corollary 2.3 on p. 474 of \cite{CCV}) with that of Bj\"orner--Frankl--Stanley ($a = (1, \ldots, 1)$ case of Theorem 1 on p. 23 and Corollary from Section 6.4 on p. 33 of \cite{BFS}) implies that $h(\Delta)$ is the $f$-vector of a balanced $\left( \frac{d}{2} - 1 \right)$-dimensional simplicial complex $\Gamma$. Note that the indexing convention for $f$-vectors on p. 24 and Section 6.4 on p. 33 of \cite{BFS} differs from the usual one for the dimensions of faces of simplicial complexes and is closer to those of cubical complexes if empty sets were to be included. We will use the usual convention for $f$-vectors of simplicial complexes by indexing the faces by dimensions instead of number of vertices. From now on, we will write $\Gamma$ to mean that $\left( \frac{d}{2} - 1 \right)$-skeleton of this simplicial complex. \\
		
		Substituting in $x = 2 \alpha + 1$, we have

		\begin{align*}
			\frac{P_\Delta(2\alpha + 1) + h_{\frac{d}{2}}(\Delta)}{2} &= \frac{P_\Delta(2\alpha + 1) + f_{\frac{d}{2} - 1}(\Gamma)}{2} \\
			&= f_{\frac{d}{2} - 1}(\Gamma) + \sum_{j = 1}^{ \frac{d}{2} } f_{ \frac{d}{2} - j - 1 }(\Gamma) (2\alpha + 1)^j \\
			&= \sum_{j = 0}^{ \frac{d}{2} } f_{ \frac{d}{2} - j - 1 }(\Gamma) (2\alpha + 1)^j
		\end{align*}

		This implies that
		
		\begin{align*}
			[\alpha^k] \frac{P_\Delta(2\alpha + 1) + h_{\frac{d}{2}}(\Delta)}{2} &= [\alpha^k] \frac{P_\Delta(2\alpha + 1) + f_{\frac{d}{2} - 1}(\Gamma)}{2} \\
			&= \sum_{j = k}^{ \frac{d}{2} }   f_{ \frac{d}{2} - j - 1 }(\Gamma) 2^k \binom{j}{k} \\
			&= 2^k \sum_{j = k}^{ \frac{d}{2} }   f_{ \frac{d}{2} - j - 1 }(\Gamma) \binom{j}{k}
		\end{align*}
		
		for $k \ge 1$ and the constant term of \[ \frac{P_\Delta(2\alpha + 1) + h_{\frac{d}{2}}(\Delta)}{2} = \frac{P_\Delta(2\alpha + 1) + f_{\frac{d}{2} - 1}(\Gamma)}{2} \] is 
		
		\begin{align*}
			[\alpha^0] \frac{P_\Delta(2\alpha + 1) + h_{\frac{d}{2}}(\Delta)}{2} &= [\alpha^0] \frac{P_\Delta(2\alpha + 1) + f_{\frac{d}{2} - 1}(\Gamma)}{2} \\
			&= \sum_{j = 0}^{ \frac{d}{2} }   f_{ \frac{d}{2} - j - 1 }(\Gamma).
		\end{align*}
		
		\color{black}
		Given $k \ge 1$, we claim that the $f$-vector of the complex $\mathcal{D}(\Gamma)$ from Definition \ref{sigunuscol} has dimensions given by the sum above. Given a $\left( \frac{d}{2} - j - 1 \right)$-dimensional face $F \in \Gamma$, there are $j$ unused colors forming the set $C_F \subset \left[ \frac{d}{2} \right]$. Choosing $k$ of them gives the subset $Q \subset C_F$ from Definition \ref{sigunuscol} and the number of pairs $(F, Q) \in \Gamma \times \left[ \frac{d}{2} \right]$ with $\dim F = \frac{d}{2} - j - 1$ and $Q \subset C_F$ is counted by the term $f_{ \frac{d}{2} - j - 1 }(\Gamma) \binom{j}{k}$ in the sum. Assigning a $\pm$ signs to each element of $Q$ gives the multiplication by $2^k$.
		
	\end{proof}

	\color{black}

	We record the (extended) $F$-polynomial definition to give some geometric context.

	\begin{defn} (Definition 3.2 on p. 578)
		Given a cell complex (Definition \ref{celldef}) $A$ of dimension $D - 1$, the \textbf{$F$-polynomial} of $A$ is \[ F_A(x) \coloneq \sum_{j = 0}^D f_{j - 1}(A) \left( \frac{x - 1}{2} \right)^j. \]
	\end{defn}

	Note that the $\widetilde{f}$-polynomial of the mirroring/Danzer complex $MA$ of $A$ has a close connection to the $F$-polynomial of $A$.

	\begin{prop} \label{Fpolytodanz}
		Given a cell complex $A$ such that $(k - 1)$-dimensional faces have $k$ vertices, we have that \[ F_A(x) = \frac{1}{2^{n(A)}} \widetilde{f}(MA, x - 1), \] where $MA$ is the mirroring/Danzer complex associated to $S$ (Definition \ref{mirdancom}). \\
		
		Equivalently, we have that \[ \widetilde{f}(MA, x) = 2^{n(A)} F_A(x + 1). \]
		
		Since $f_K(t) = t \widetilde{f}(K, t) + 1$ for a cubical complex $K$ by the discussion in Remark \ref{indexcon}, this implies that \[ f(MA, x) = 2^{n(A)} x F_A(x + 1) + 1 \] for such a cell complex $A$.
	\end{prop}

	\begin{rem}
		We can write a generalization applying to arbitary cell complexes using Proposition \ref{posetfdanzer} if we replace the $\widetilde{f}$-polynomial with the $\widetilde{S}$-polynomial parametrized by the number of faces with a given number of \emph{vertices} (instead of dimension). \\
	\end{rem}
	
	\color{black}

	\begin{proof}
		In spite of the indexing issues discussed in Remark \ref{indexcon}, the $\widetilde{f}$-polynomial of a cell complex from Definition \ref{posetfpoly} is the same as its usual $f$-polynomial. Now recall from Proposition \ref{posetfdanzer} that \[ \widetilde{f}(MT, t) = 2^n \widetilde{f} \left( T, \frac{t}{2} \right) \] for a cell complex $T$ with $n$ vertices such that $(k - 1)$-dimensional faces have $k$ vertices. \\
		
		Since $F_A(x) = f_A \left( \frac{x - 1}{2} \right)$ for a cell complex $A$, this means that
		
		\begin{align*}
			\widetilde{f}(MA, t) &= 2^n \widetilde{f} \left( A, \frac{t}{2} \right) \\
			&= 2^n f \left( A, \frac{t}{2} \right) \\
			\Longrightarrow \widetilde{f}(MA, x - 1) &= 2^n f \left( A, \frac{x - 1}{2} \right) \\
			&= 2^n F_A(x) \\
			\Longrightarrow F_A(x) &= \frac{1}{2^n} \widetilde{f}(MA, x - 1)
		\end{align*}
		
		after substituting in $t = x - 1$.
		
	\end{proof}

	The main application of the $F$-polynomial of a simplicial complex $\Delta$ in \cite{He2} is its compatibility with Tchebyshev transformations $T : \mathbb{R}[x] \longrightarrow \mathbb{R}[x]$ sending $x^n \mapsto T_n(x)$ (the $n^{\text{th}}$ Chebyshev polynomial of the first kind) since $T(F(\Delta))(x) = F_{T(\Delta)}(x)$ (Proposition 3.3 on p. 578 of \cite{He2}). As it turns out, the definition of a Tchebyshev subdivision and the formula for its $f$-polynomial (in the usual sense) do \emph{not} depend on $\Delta$ being a simplicial complex. We will list the generalizations below and comment on why this is the case for the relevant definition and result. \\

	\begin{defn} \textbf{(Cell complex generalization of Definition 2.1 on p. 573 of \cite{He2})} \label{Tchebdivgen} \\
		
		Let $A$ be (the dual of) a cell complex from Definition \ref{celldef}. \\
		
		Given a cell complex $A$ on a vertex set $V$, fix a linear order $<$ on $V$ and introduce a new element $\widehat{1}$ larger than all elements of $V$. The \textbf{Tchebyshev triangulation $T_< A$} of $A$ is the cell complex on the vertex set \[ \{ (u, v) : u, v \in V, u < v, \{ u, v  \} \subset F \text{ for some } F \in A  \} \cup \{ (u, \widehat{1}) : u \in V \} \] 
		
		whose $(k - 1)$-faces are all sets $\{ (u_1, v_1), \ldots, (u_k, v_k) \}$ satisfying $v_1 < \cdots < v_k$ and the following conditions:
		\begin{enumerate}
			\item The set $\{ u_1, v_1, u_2, v_2, \ldots, u_k, v_k \} \setminus \{ \widehat{1} \}$ is a face of $A$. 
			
			\item  For all $i \le k - 1$, either $u_i = u_{i + 1}$ or $v_i \le u_{i + 1}$ holds. 
		\end{enumerate}
		
		Since the $f$-vector of $T_<(A)$ is independent of the ordering $<$, we will often drop the ordering and simply write $T(A)$. \\
	\end{defn}

	\begin{rem} \textbf{(Connection between definitions and comments on generalization to cell complexes) \\} \label{tchebdivgencom}
		\vspace{-3mm}
		\begin{enumerate}
			\item This is an adaptation of Definition 2.1 on p. 573 of \cite{He2} by Hetyei which specializes to the original definition when $A$ is a simplicial complex. We can think of pairs of actual vertices of $A$ as vertices subdividing the corresponding edges and the pairs ending with $\widehat{1}$ correspond to vertices of the original cell complex $A$. The remaining conditions ensure that the ``realization'' lies inside some face of $A$ and we're moving forward with respect to an edge ordering. Note that we think about faces of the subdivision in terms of which ``side'' of each edge we're on. This is used later in the $f$-vector count of Tchebyshev subdivisions and is reminiscent of how one thinks about faces in barycentric subdivisions. \\

			While a more intuitive definition matching the description above more directly is given in Definition 2 on p. 921 of \cite{He3}, we use the one above since it generalizes to arbitary cell complexes. For example, 1-dimensional faces in the dual of a cell complex may have more than one vertex unless the starting cell complex gives rise to a pseudomanifold without boundary (in which case every ridge is contained in 2 facets in the original cell complex -- see Definition 3.15 on p. 24 of \cite{St}). \\
			
			\item  While each face $G$ of $T(A)$ of dimension $k - 1$ has the ``right'' number of vertices to be a $(k - 1)$-dimensional face of a simplicial complex, $T(A)$ itself isn't necessarily a simplicial complex. This is because removing the underlying set of a vertex of $G$ (given by a singleton or a pair) from the face of $A$ given by the union of the underlying sets of the vertices of $G$ might not yield a face of $A$ if $A$ isn't a simplicial complex. \\
			
			\item We have that $\dim T(A) + 1$ is the maximum number of vertices of a face of $A$. This is because a union of vertices of $A$ coming from $k$ vertices of the form $(u_i, v_i)$ has $m \in [k, 2k]$ elements (e.g. see the proof of Theorem 3.1 on p. 577 -- 578 of \cite{He2}). For example, Definition \ref{sigunuscol} implies that $\dim T(\mathcal{D}(\Gamma)) = \frac{d}{2} - 1$ since any face of $\mathcal{D}(\Gamma)$ has $\le \frac{d}{2}$ vertices. \\ 
		\end{enumerate}
	\end{rem}

	\begin{thm} \textbf{(Cell complex generalization of Theorem 3.1 on p. 577 and Proposition 3.3 on p. 578 -- 579 of \cite{He2} by Hetyei)} \label{Ftchebgen} \\
		Let $A$ be a cell complex (Definition \ref{celldef}). \\
		\begin{enumerate}
			\item (Theorem 3.1 on p. 577 -- 578 of \cite{He2}) \\
			For $k \ge 1$, the number of $(k - 1)$-faces of of $T(A)$ is given by \[ f_{k - 1}(T(A)) = \sum_{j = k}^{2k} f_{j - 1}(A) \cdot 2^{2k - j - 1} \left(  \binom{k}{2k - j} + \binom{k - 1}{2k - j}  \right). \]
			
			\item  (Proposition 3.3 on p. 578 -- 579 of \cite{He2}) \\
			The Tchebyshev transform of the $F$-polynomial of $A$ is the $F$-polynomial of its Tchebyshev triangulation: \[ T(F_A)(x) = F_{T(A)}(x). \]
			
		\end{enumerate}
	\end{thm}

	\begin{rem} \textbf{(Comments on generalizations) \\}
		\vspace{-3mm}
		\begin{enumerate}
			\item For Part 1, the proof of Theorem 3.1 on p. 577 -- 578 of \cite{He2} doesn't use anything about the property defining an abstract simplicial complex (subsets of a face of the given complex being a face). In other words, the only input comes from the ordering properties defining Tchebyshev subdivisions and pairs of vertices contained in a given face of $A$ from Definition \ref{Tchebdivgen}. After starting out by selecting a face to use for the set \[ \{ u_1, v_1, u_2, v_2, \ldots, u_k, v_k \} \setminus \{ \widehat{1} \}, \] \color{black} the rest of the work has to do what terms to pair together and in what order the pairs should be listed. The main structure of the proof has to do with the smallest face of the original cell complex $A$ that ``contains'' the given face of the Tchebyshev subdivision $T(A)$ and keeping track of possible slots for $u_i$'s and $v_j$'s (labeled with $+$ and $-$ signs), and how many strings with a given set of signs satisfy condition b) in Part 1 of Definition \ref{Tchebdivgen}. This is sort of similar to how we keep track of faces on a barycentric subdivision. \\
			
			\item The identity in Part 2 is purely a consequence of the formula in Part 1.
		\end{enumerate}
		
	\end{rem}
	
	Combining these generalizations with Theorem \ref{danzinput} implies the following:
	
	\begin{cor} \label{invtchebsub}
		Recall that \[ g_\Delta(u) \coloneq h_{ \frac{d}{2} }(\Delta) + 2 \sum_{j = 1}^{ \frac{d}{2} } h_{ \frac{d}{2} - j }(\Delta) T_j \left( \frac{u}{2} \right). \]
		
		We have that \[ \frac{g_\Delta(2w) + f_{ \frac{d}{2} - 1 }(\Gamma) }{2} - \sum_{i = 0}^{ \frac{d}{2} } f_{ \frac{d}{2} - j - 1 }(\Gamma) + 1  = F(T(\mathcal{D}(\Gamma)), w) = \frac{1}{2^{ n(T(\mathcal{D}(\Gamma))) }  } \widetilde{f}(M(T(\mathcal{D}(\Gamma)), w - 1), \] where the last equality follows from Proposition \ref{Fpolytodanz}. \\
	\end{cor}

	\color{black}

	We will now consider the Tchebyshev triangulation (Definition \ref{Tchebdivgen}) of the signed unused color complex from Definition \ref{sigunuscol}.

	\color{black}
	
	Going back to the sums under consideration, recall from Corollary \ref{invtchebsub} that
	
	\begin{align*}
		\frac{g_\Delta(2w) + f_{ \frac{d}{2} - 1 }(\Gamma) }{2} - \sum_{i = 0}^{ \frac{d}{2} } f_{ \frac{d}{2} - i - 1 }(\Gamma) + 1 &= f \left( T(\mathcal{D}(\Gamma)), \frac{w - 1}{2} \right) \\
		&= \frac{1}{2^{n(T(\mathcal{D}(\Gamma)))}} \widetilde{f}(M(T(\mathcal{D}(\Gamma))), w - 1).
	\end{align*}

	The coefficient of the degree $k$ term in \[ R(w) \coloneq f \left( T(\mathcal{D}(\Gamma) ), \frac{w}{2} \right) \] is $f_{k - 1}(T( \mathcal{D}(\Gamma) ) ) 2^{-k}$. A $k$-tuple of vertices of $T(\mathcal{D}(\Gamma))$ forming a face of $\mathcal{D}(\Gamma)$ has between $k$ and $2k$ choices of signs to make for each vertex of the Tchebyshev subdivision from possible sizes in a cell complex $A$ of the union of vertices of $A$ coming from $k$ vertices of of the Tchebyshev subdivision $T(A)$ (analogous to considerations in the proof of Theorem 3.1 on p. 577 -- 578 of \cite{He2}). This bound comes from having $2$ sign choices for every ``valid'' singleton of the form $(u, \widehat{1})$ from Definition \ref{Tchebdivgen} and $4$ sign choices when we have a valid pair $(u, v)$. This means that we have $2^m$ sign choices for some $k \le m \le 2k$. Effectively, what this means is that we're considering $k$-tuples consisting of unsigned pairs of the form $(u, \widehat{1})$ and pairs $(u, v)$ where only the second coordinate is signed. Equivalently, $u$ and $v$ must have the same sign. This is a sort of ``complex of one-sided intervals'' $\widetilde{T}(\mathcal{D}(\Gamma))$ (a modification of the Tchebyshev subdivision $T(\mathcal{D}(\Gamma))$ of $\mathcal{D}(\Gamma)$ which we define below (with relevant parts reproduced from Definition \ref{Tchebdivgen}). \\

	\begin{defn}
		A vertex $v = (B \subset \{ q \} \subset C_F)$ of $\mathcal{D}(\Gamma)$ has \textbf{positive} sign if $B = \{ q \}$ and it has \textbf{negative} sign if $B = \emptyset$. We denote these cases using the notation $\sgn(v) = 1$ and $\sgn(v) = -1$ respectively.
	\end{defn}
	
	\begin{itemize}
		\item The vertex set of $\widetilde{T}(\mathcal{D}(\Gamma))$ is \[ \{ (u, v) \in V(\mathcal{D}(\Gamma)) : u < v, \{ u, v  \} \subset F \text{ for some } F \in \mathcal{D}(\Gamma) \] and \[ \sgn(u) = \sgn(v) \} \cup \{ (u, \widehat{1}) : u \in V(\mathcal{D}(\Gamma)) \}. \]
		
		\item The $(k - 1)$-faces of $\widetilde{T}(\mathcal{D}(\Gamma))$ are all the sets $\{ (u_1, v_1), \ldots, (u_k, v_k) \}$ satisfying $v_1 < \cdots < v_k$ and the following conditions:
		\begin{enumerate}
			\item The set $\{ u_1, v_1, u_2, v_2, \ldots, u_k, v_k \} \setminus \{ \widehat{1} \}$ is a face of $\mathcal{D}(\Gamma)$.
			
			\item  For all $i \le k - 1$, either $u_i = u_{i + 1}$ or $v_i \le u_{i + 1}$ holds. \\
		\end{enumerate}
	\end{itemize}
	
	\color{black}

	\color{black}

	The expression in Remark \ref{fmatchstrat} is \[ g_\Delta(2w) = \sum_{j = 0}^{ \frac{d}{2} } f_{ \frac{d}{2} - j - 1 }(T) 2^j (w + 1)^j = \sum_{j = 0}^{ \frac{d}{2} } f_{ j - 1 }(T^*) 2^j (w + 1)^j. \]
	
	In some sense, this indicates that the mirroring/Danzer complex of the Tchebyshev subdivision $T(\lk_\Delta(e_k) \big\backslash \left( \left( \bigcup_{i = 1}^{k - 1} \lk_\Delta(e_i) \right) \cap \lk_\Delta(e_k) \right))$ of $\mathcal{D}(\Gamma)$ is the interval poset of the simplicial complex $T$ such that $\gamma(\Delta) = f(T)$. \\
	
	This would be equivalent to having \[ [w^k] g_\Delta(2w) = \sum_{j = k}^{ \frac{d}{2} } f_{j - 1}(T^*) 2^j \binom{j}{k} \]

	Note that the terms above also count signed intervals where the upper bound is assigned a sign and the lower one (from a subset of size $k$) is unsigned (equivalently inheriting signs from the top one). \\

	So far, we have shown that 
	
	\begin{align*}
		[w^k] \left(  \frac{g_\Delta(2w) + f_{ \frac{d}{2} - 1 }(\Gamma) }{2} - \sum_{i = 0}^{ \frac{d}{2} } f_{ \frac{d}{2} - i - 1 }(\Gamma) + 1  \right) &= [w^k] F(T(\mathcal{D}(\Gamma) ), w) \\ 
		&=  [w^k] R(w - 1) \\ 
		&= [w^k] A(w),
	\end{align*} 
	
	where $A$ denotes the polynomial $A(w)$ such that $R(w) = A(w + 1)$ ($R(w)$ itself defined below Corollary \ref{invtchebsub}). We can take $A(w)$ to have nonnegative coefficients if $\gamma(\Delta) = f(T)$. \\
	
	If we substitute $w = \mu + 1$, we have
	
	\begin{align*}
		[\mu^k] \left(  \frac{g_\Delta(2(\mu + 1)) + f_{ \frac{d}{2} - 1 }(\Gamma) }{2} - \sum_{i = 0}^{ \frac{d}{2} } f_{ \frac{d}{2} - i - 1 }(\Gamma) + 1  \right) &= [\mu^k] F(T(\mathcal{D}(\Gamma) ), \mu + 1) \\ 
		&=  [\mu^k] R(\mu) \\ 
		&= [\mu^k] A(\mu + 1) \\ 
		&= \sum_{j = k}^{ \frac{d}{2} } A_{j - 1} \binom{j}{k},
	\end{align*}
	
	where $A_{j - 1}$ denotes the coefficient of $x^j$ in $A(x)$. Note that \[ [\mu^k] R(\mu) = [\mu^k] f \left( T(\mathcal{D}(\Gamma) ), \frac{\mu}{2} \right) = f_{k - 1}(T(\mathcal{D}(\Gamma))) 2^{-k} = f_{k - 1}(\widetilde{T}(\mathcal{D}(\Gamma))) \] and 
	
	\begin{align*}
		g_\Delta(2w) &= (2w + 2)^{ \frac{d}{2} } \gamma_\Delta \left( \frac{1}{2w + 2} \right) \\ 
		\Longrightarrow  g_\Delta(2(\mu + 1)) &= (2(\mu + 1) + 2)^{ \frac{d}{2} } \gamma_\Delta \left( \frac{1}{2(\mu + 1) + 2} \right) \\ 
		&= (2(\mu + 2))^{ \frac{d}{2} } \gamma_\Delta \left( \frac{1}{2(\mu + 2)} \right) \\
		&= (2\mu + 4)^{ \frac{d}{2} } \gamma_\Delta \left( \frac{1}{2\mu + 4} \right) .
	\end{align*}
	
	After multiplying by 2, we have 
	
	\begin{align*}
		[\mu^k] \left(  g_\Delta(2(\mu + 1)) + f_{ \frac{d}{2} - 1 }(\Gamma) - 2\sum_{i = 0}^{ \frac{d}{2} } f_{ \frac{d}{2} - i - 1 }(\Gamma) + 1  \right) &= 2 [\mu^k] F(T(\mathcal{D}(\Gamma) ), \mu + 1) \\ 
		&= 2 [\mu^k] R(\mu) \\ 
		&= 2 \cdot 2^k [\mu^k] f(\widetilde{T}(\mathcal{D}(\Gamma)), \mu) \\
		&= 2 \cdot 2^k [\mu^k] \widetilde{A}(\mu + 1) \\ 
		&= 2 \cdot 2^k \sum_{j = k}^{ \frac{d}{2} } \widetilde{A}_{j - 1} \binom{j}{k},
	\end{align*}

	where $\widetilde{A}(x) = f(\widetilde{T}(\mathcal{D}(\Gamma)), x - 1)$ and $\widetilde{A}_{j - 1}$ denotes the coefficient of $x^j$ in $\widetilde{A}(x)$. The multiplication by $2^k$ means that we put $\pm$ signs on each face (alternatively put signs on an extraneous element $\widehat{1}$ for each face). \\
	
	\color{black}
	Suppose that $\widetilde{A}(x)$ has nonnegative coefficients. To obtain a cell complex $T^*$ where \[ [w^0] g_\Delta(2w) = 2 \sum_{j = 0}^{ \frac{d}{2} } f_{j - 1}(A) +  \left(f_{ \frac{d}{2} - 1 }(\Gamma) + 2 \sum_{i = 1}^{  \frac{d}{2} - 1 } f_{  \frac{d}{2} - i - 1 }(\Gamma)  \right) = \sum_{j = 0}^{ \frac{d}{2} } f_{j - 1}(T^*), \]
	
	it suffices to formally ``add'' $\left( f_0(\Gamma^*) + 2 \sum_{i = 1}^{  \frac{d}{2} - 1 } f_i(\Gamma^*) \right)$ disjoint points since $f_0(T^*)$ does \emph{not} appear in the sums determining $[w^k] g_\Delta(2w)$ for any $k \ge 1$. \\
	
	\begin{cor} \label{simpgamdanz} ~\\
		\vspace{-3mm}
		\begin{enumerate}

			\item Suppose that $h(\Delta)$ is the $h$-vector of a flag sphere of odd dimension $d - 1$ (i.e. $d$ is even). If $d \ge 4$, then the polynomial $g_\Delta(2(u + 1)) = (2(u + 2))^{\frac{d}{2}} \gamma_\Delta \left( \frac{1}{2(u + 2)} \right)$ has nonnegative coefficients. After multiplying by a power of 2, we have the $f$-polynomial of the mirroring/Danzer/power complex $M(T(\mathcal{D}(\Gamma)))$ with a nonnegative constant added to it. \\
			
			\item By the $f_k^\Box(S)$ formula below Definition \ref{cubbary}, substituting the variable $w$ by $u + 1$ sends the $f$-polynomial of a simplicial complex to that of its cubical barycentric subdivision. Consider the $f$-polynomial of the mirroring/Danzer/power complex $M(T(\mathcal{D}(\Gamma)))$. \\
			
			Corollary \ref{invtchebsub} implies that its $f$-polynomial behaves ``formally''  like a double cubical barycentric subdivision associated to some integer tuple. If $\gamma(\Delta) = f(T)$, we can remove the word ``formal'' and take the $f$-polynomial associated to a cell complex after a rescaling as the input.  \\

		\end{enumerate}
		
		\color{black}

	\end{cor}
	
	\color{black}

	Finally, we study some geometric properties connected to nonnegativity of translation by $-1$ (thinking about the gamma vector). Given $MA$, note that the links of the vertices are isomorphic to $A$. The following result shows that $MA$ is a $CAT(0)$ cubical complex if and only if $A$ is a flag simplicial complex and $MA$ is simply connected:
	
	\begin{thm} (Gromov, Section 4 of \cite{Gro} and Theorem 2.3 of \cite{Row2}) \\
		A cubical complex $\Box$ is $CAT(0)$ if and only if it is simply connected and all the vertex links are flag simplicial complexes. Moreover, any $CAT(0)$ cubical complex is contractible.
	\end{thm}
	
	The results below imply that $\widetilde{f}(M(T(A)), w) = P(w + 1)$ for some polynomial $P(z)$ with nonnegative coefficients if $T(A)$ is flag.
	
	\begin{lem} (Rowlands, Definition 3.1 on p. 10 and Lemma 3.3 on p. 11 of \cite{Row2}) \\
		Every flag simplicial complex is the crossing complex $\Delta_P$ (associated to a $PIP$ $P$) of some $CAT(0)$ cubical complex.
	\end{lem}
	
	\begin{thm} (Ardila--Owen--Sullivant, Theorem 2.5 on p. 146 of \cite{AOS}, Theorem 2.9 of \cite{Row2}) \\
		The map $P \mapsto \Box_P$ is a bijection from the set of $PIP$s to the set of rooted $CAT(0)$ cubical complexes, with inverse given by $\Box \mapsto P_\Box$. 
	\end{thm}
	
	\begin{thm} (Rowlands, Theorem 3.5 on p. 13 of \cite{Row2}) \\
		Let $P$ be a $PIP$. Then, $\widetilde{f}(\Box_P, t) = \widetilde{f}(\Delta_P, 1 + t)$. \\
	\end{thm}

	In the context of Definition \ref{cubbary}, a natural question would be whether $\Box_P$ behaves similarly to a cubical barycentric subdivision of $\Delta_P$. Input simplicial complexes/cell complexes $A$ such that $T(A)$ is flag yield outputs connected to gamma vectors which arise from $f$-vectors of flag simplicial complexes. Note that the $h$-vector of any flag Cohen--Macaulay simplicial complex is the $f$-vector of an auxiliary simplicial complex. This auxiliary simplicial complex can be taken to be $d$-chromatic (1-skeleton has a proper coloring with $d$ colors, see Section 6.4 on p. 33 of \cite{BFS}) and balanced when $h_d(\Delta) \ne 0$ (e.g. when the Dehn--Sommerville relations $h_i(\Delta) = h_{d - i}(\Delta)$ are satisfied). In addition, it was conjectured earlier (e.g. (Conjecture 3.11 on p. 94 of \cite{CN})) that the $h$-vector of every vertex decomposable flag simplicial complex is the face vector of some flag complex . \\

	 \section{Recursive properties involving the $f$-vector perspective based on older decompositions} \label{gamfdec}

	 In Section \ref{approachback}, an important input for connecting the inverted Chebyshev expansion associated to the gamma polynomial of a flag sphere $\Delta$ was the the property $h(\Delta) = f(\Gamma)$ for an auxiliary simplicial complex $\Gamma$. Given this context, it is natural to consider recursive properties of flag simplicial complexes $\Delta$ such that $h(\Delta) = f(\Gamma)$. They will involve common structures occurring in vertex decomposability and edge subdivisions. \\ 
	 
	 To this end, we will consider Boolean decompositions. Note that a property that was studied in the context of possible symmetric Boolean decompositions (see Section 4.7 on p. 84 -- 85 of \cite{Pet}). Another instance where Boolean decompositions appear is in possible structures of simplicial complexes whose $f$-vectors are $h$-vectors of gamma-nonnegative flag spheres. \\

	 \color{black}

	 \begin{prop} (Nevo--Petersen--Tenner, Proposition 6.2 on p. 1377 of \cite{NPT}) \label{fvectbool} \\
	 	If $\gamma(\Delta)$ is the $f$-vector of a simplicial complex $S$, then $h(\Delta)$ is the $f$-vector of the following simplicial complex:
	 	
	 	\[ \Gamma = \{ F \cup G : F \in \mathcal{F}(\gamma(\Delta)), G \in 2^{[d - 2|F|]} \}, \] where $\mathcal{F}(T)$ is the standard compressed complex associated to $T$ taking $\frac{d}{2}$-compressed complexes $\mathcal{F}_k(f(T))$ and letting $k \to \infty$ (see p. 263 of \cite{Zie} and p. 1366 -- 1369 of \cite{NPT}).
	 	
	 	Note that the part involving $|F|$ is multiplied by $2$. \\
	 \end{prop}
	 
	 Since the identity expressing $h_i$ as a linear combination of $\gamma_j$ used to prove the result above (Observation 6.1 on p. 1377 of \cite{NPT}) is an invertible linear transformation, having such a decomposition for the simplicial complex $\Gamma$ such that $h(\Delta) = f(\Gamma)$ (which exists by work of Caviglia--Constantinescu--Varbaro \cite{CCV} and Bj\"orner--Frankl--Stanley \cite{BFS}) is equivalent to the gamma vector of $\Delta$ being the $f$-vector of some simplicial complex. This structure can be recorded in a definition. \\
	 
	 \begin{defn} \label{booldecomp}
	 	A simplicial complex $\Gamma$ of dimension $d - 1$ has a \textbf{Boolean decomposition} if it an be expressed as \[ \Gamma = \{ F \cup G : F \in S, G \in 2^{[d - 2|F|]} \} \] for some simplicial complex $S$.
	 \end{defn}
	 
	 We study how this decomposition behaves under transformations often used on flag simplicial complexes applied to $\Gamma$ or $\Delta$. The result below is one indication that the latter is also natural to consider. We focus on this vertex decomposable ($k = 1$ case of Defintiion 2.1 on p. 579 of \cite{PB}, Definition 11.1 on p. 3965 of \cite{BWa2}) case involving a variation of the flagness property (Definition 4.2 on p. 99 and Remark 3.5 on p. 97 of \cite{CV}) since there are concrete recursive properties to study. \\
	 
	 \begin{thm} (Constantinescu--Varbaro, Theorem 4.3 on p. 99 of \cite{CV}) \\
	 	Let $\Delta$ be a balanced, vertex decomposable, flag simplicial complex on $[n]$. Then there exists a quasi-flag simplicial complex $\Gamma$ such that $f(\Gamma) = h(\Delta)$. \\
	 \end{thm}

	 In addition, the $f$-vector of a flag simplicial complex $\Gamma$ is the $h$-vector of some vertex decomposable flag complex $\Delta$ by work of Cook--Nagel (Corollary 3.10 on p. 94 of \cite{CN}). Conjecturally, the $h$-vector of every vertex decomposable flag complex is the $f$-vector of some flag complex (Corollary 3.11 on p. 94 of \cite{CN}). We consider how possible structures on $\Gamma$ compare with those of other flag simplicial complexes that are PL homeomorphic to it. By a result of Lutz--Nevo (Theorem 1.2 on p. 70 and p. 77 of \cite{LN}), it suffices to consider edge subdivisions and contractions of $\Gamma$. \\
	 
	 This motivates a generalization of the Boolean decompositions.
	 
	 \begin{defn} \label{genbool}
	 	A \textbf{generalized Boolean decomposition} of a simplicial complex $\Gamma$ of dimension $d - 1$ is a decomposition of $\Gamma$ as the disjoint union of collections of faces of the form \[ \{ F \cup G : F \in S, G \subset 2^{[d - 2|F|]} \} \] where $S$ is either a simplicial complex or the open star of one and $G$ is a subcollection of the $2^{[d - 2|F|]}$ restricted by size (i.e. faces $F \cup G$ with $|F|$ lying in a specified interval) or containing a specific element (similar to the open star). Here, the ``open star'' of a simplicial complex $\Delta$ means the faces of $\Delta$ containing a particular subset $R \subset V(\Delta)$ of the vertex set $V(\Delta)$ (p. 7 of \cite{HSL}).
	 \end{defn}
	 
	 We use this to show the following:
	 
	 \begin{prop} \label{edgegenbool}
	 	The edge subdivision of a simplicial complex that has a Boolean decomposition has a generalized Boolean decomposition.
	 \end{prop}
	 
	 \begin{proof}
	 	Let $d = \dim \Gamma + 1$ be an even number. Suppose there is some simplicial complex $S$ of dimension $\frac{d}{2} - 1$ such that $\Gamma = \{ F \cup G : F \in S, G \in 2^{[d - 2|F|]} \}$. Let $e = \{ a, b \} \in \Gamma$ be an edge of $\Gamma$. We consider where there is such a decomposition for the edge subdivision $\Gamma'$ of $\Gamma$ with respect to $e$. Let $v \in V(\Gamma')$ be the vertex subdividing what was the edge $e$ in $\Gamma$. In general, we have that $\Ast_{\Gamma'}(v) \cong \Ast_\Gamma(e)$ while faces of $\Gamma$ that contained $e$ are doubled by $v$ and $\lk_\Gamma(e) \cong \lk_{\Gamma'}(a, v) \cong \lk_{\Gamma'}(b, v)$ (with $v$ replacing one of $a$ or $b$). We split this into 3 cases: \\
	 	
	 	\begin{enumerate}
	 		\item Suppose that $a, b \in [d]$ (i.e. the Boolean part $G$ of the decomposition). Without loss of generality, we will assume that $a < b$. Note that there is an extra new vertex $v$ involved in the ``Boolean'' part. Then, $S$ and faces of the form $F \cup G$ with $d - 2|F| \le a - 1$ or $d - 2|F| \ge b + 1$ (equivalently $|F| \ge \frac{d - a + 1}{2}$ or $|F| \le \frac{d - b - 1}{2}$) in general are unaffected. This is because the subdivided faces are those with $a \le d - 2|F| \le b$. Then, the faces $F \in S$ yielding doubled faces $F \cup G$ are those with \[ \frac{d - b}{2} \le |F| \le \frac{d - a}{2}. \] The added part has a decomposition into two parts with the first part from $F \in S$ such that $\frac{d - b}{2} \le |F| \le \frac{d - a}{2}$ and subsets of $[d - 2|F|]$ such that $b \le d - 2|F|$ (equivalently $\frac{d - b}{2} \le |F|$). \\
	 		
	 		
	 		\item Now assume that $a \in V(S)$ and $b \in [d]$. The faces of $\Gamma$ affected are those using $a \in V(S)$ and $b \in [d]$. These come from faces $F \in S$ containing $a$ and subsets $G \in 2^{[d - 2|F|]}$ containing $b$. In other words, the second condition states that $d - 2|F| \ge b$ (equivalently $b \le d - 2|F|$ or $|F| \le \frac{d - b}{2}$). For these faces, the subdividing vertex $v$ doubles the faces by replacing $a$ with $v$ or $b$ with $v$. The added part (faces of $\Gamma'$ not counted by $\Gamma$) has something similar to a Boolean decomposition where we consider faces $F \cup G \in \Gamma$ using $F \in S$ such that $F \ni a$ and subsets of $[d - 2|F|]$ containing $b$ (i.e. $|F| \le \frac{d - b}{2}$) instead of all of $[d - 2|F|]$. \\
	 		
	 		\item If $a, b \in V(S)$, then this is essentially an edge subdivision $S'$ of $S$ using the same $S$. The components/faces $G \in 2^{[d - 2|F|]}$ used remain the same and we still have a Boolean decomposition without splitting the new simplicial complex $\Gamma'$ into parts. 
	 	\end{enumerate}
	 	
	 \end{proof}
	 

	 We now move to the other side and take $\Delta$ to be some flag Cohen--Macaulay simplicial complex. Write $h(\Delta) = f(\Gamma)$ for some simplicial complex $\Gamma$ (which exists by work of Caviglia--Constantinescu--Varbaro \cite{CCV} and Bj\"orner--Frankl--Stanley \cite{BFS}). This simplicial complex $\Gamma$ can be taken to be $d$-chromatic (i.e. 1-skeleton has a proper $d$-coloring of vertices, see Section 6.4 on p. 33 of \cite{BFS}), which means that it is balanced if $h_d(\Delta) \ne 0$ (e.g. if $\Delta$ satisfies the Dehn--Sommerville relations). Given any flag simplicial complex $\Gamma$, there is a balanced flag vertex decomposable simplicial complex $\Delta$ such that $h(\Delta) = f(\Gamma)$ by a result of Constantinescu--Varbaro (Proposition 4.1 on p. 98 -- 99  of \cite{CV}). Let $\Delta'$ be the (stellar) edge subdivision of $\Delta$ at an edge $e \in \Delta$. Suppose that there are simplicial complexes $\Gamma$ and $\Gamma_e$ with Boolean decompositions \[ \Gamma = \{ F \cup G : F \in S, G \in 2^{ [d - 2|F|] } \} \] and \[ \Gamma_e = \{ F_e \cup G_e : F_e \in S_e, G_e \in 2^{ [d - 2|F_e| - 2] } \} \] involving simplicial complexes $S$ of dimension $\frac{d}{2} - 1$ and $S_e$ of dimension $\frac{d}{2} - 2$ such that $f(\Gamma) = h(\Delta)$ and $f(\Gamma_e) = h(\lk_\Delta(e))$. By the invertibility of the linear transformation in Observation 6.1 on p. 1377 of \cite{NPT}, this would mean that $\gamma_i(\Delta) = f_{i - 1}(S)$ and $\gamma_j(\lk_\Delta(e)) = f_{j - 1}(S_e)$. Note that $\gamma_{\Delta'}(t) = \gamma_\Delta(t) + t \gamma_{\lk_\Delta(e)}(t)$. In other words, we have that $\gamma_{\Delta'}(t) = f_S(t) + t f_{S_e}(t)$. This means that $\gamma_i(\Delta') = f_{i - 1}(S) + f_{i - 2}(S_e)$. We would like to relate this to underlying simplicial complex structures on the left hand side.   \\
	 
	 Since $\Delta$ is Cohen--Macaulay, we have that $h_i(\lk_\Delta(e)) \le h_i(\Delta)$ (Corollary 9.2 on p. 127 of \cite{St}). This means that $f_{i - 1}(S_e) \le f_{i - 1}(S)$ and we can treat $S_e$ as a subcomplex of $S$ by using compression complexes. If we set $\Gamma' = \Gamma *_{\Gamma_e} u$ for some new vertex $u$, we have that $f_{i - 1}(\Gamma') = f_{i - 1}(\Gamma) + f_{i - 2}(\Gamma_e)$ (see p. 97 of \cite{CV}). In fact, we can show that this partial coning construction is compatible with the Boolean decomposition in addition to having matching $f$-vectors. We note that the source of this partial coning comes from the construction for simplicial complexes whose $f$-vectors yield $h$-vectors of flag vertex decomposable simplicial complexes from Theorem 3.3 on p. 96 -- 97 of \cite{CV}. Going backwards from $\Gamma'$ to $\Gamma$ is the same as removing $u$ (equivalently identifying $u$ with some vertex of $\Gamma_e$). This is explained in further detail below. \\ 

	 
	 
	 \begin{prop}  \label{commrec}
	 	Let $\Gamma_1$ be a simplicial complex of dimension $D - 1$ and $\Gamma_2 \le \Gamma_1$ be subcomplex of dimension $D - 2$. Let $u$ be a new disjoint vertex. Note that we can instead take $\Gamma_2$ to be a simplicial complex such that $f_i(\Gamma_2) \le f_i(\Gamma_1)$ for all $i$ since we can take compression complexes. If $\Gamma_1$ and $\Gamma_2$ have Boolean decompositions, the simplicial complex \[ \Gamma \coloneq \Gamma_1 *_{\Gamma_2} u = \Gamma_1 \cup (\Gamma_2 * u) \] also has a Boolean decomposition. 
	 \end{prop}
	 
	 \begin{proof}
	 	Let \[ \Gamma_1 = \{  F_1 \cup G_1 : F_1 \in S_1, G_1 \in 2^{[d - 2|F_1|]} \} \] and \[ \Gamma_2 = \{ F_2 \cup G_2 : F_2 \in S_2, G_2 \in 2^{[d - 2|F_2| - 1]} \} \] be Boolean decompositions for $\Gamma_1$ and $\Gamma_2$. Then, we have that $\Gamma_2 * u$ has the Boolean decomposition given by \[ \Gamma_2 * u = \{ F_2 \cup \widetilde{G}_2 : F_2 \in S_2, \widetilde{G}_2 \in 2^{[d - 2|F_2|]} \}, \] where the new vertex $u$ plays the role of the vertex $d$ of $2^{[d]}$ that was ``missing'' in the Boolean decomposition of $\Gamma_2$. The modification of the Boolean part $G_2$ used for $\Gamma_2$ adding this new vertex to the Boolean part is denoted $\widetilde{G}_2$ above. \\
	 	
	 	This yields the Boolean decomposition \[ \Gamma = \{  F \cup G : F_1 \in S_1 \cup S_2, G \in 2^{[d - 2|F|]} \} \] for $\Gamma$. \\
	 \end{proof}
	 
	 
	 We can use this result and to apply the reasoning used in the edge subdivision case to geometrically natural examples. Given a simplicial complex $\Delta$, the key properties we use for inductive purposes are that the flag, vertex decomposable, and balanced properties are also satisfied by $\Ast_\Delta(v)$ and $\lk_\Delta(v)$ if they are satisfied by $\Delta$. Note that these are main recursive inputs in the proofs of Theorem 3.3 on p. 96 -- 97 and Theorem 4.3 on p. 99 -- 100 of \cite{CV}. \\
	 
	 \begin{cor} \label{vertdecbool} ~\\
	 	\vspace{-3mm}
	 	\begin{enumerate}
	 		\item Suppose that a simplicial complex $\Delta$ has an associated simplicial complex $\Gamma$ such that $h(\Delta) = f(\Gamma)$ and $\Gamma$ has a Boolean decomposition. Let $\Delta'$ be the (stellar) subdivision of $\Delta$ with respect to an edge $e \in \Delta$. Then, there is a simplicial complex $\Gamma'$ such that $h(\Delta') = f(\Gamma')$ and $\Gamma$ has a generalized Boolean decomposition. If $\lk_\Delta(e)$ also has $h(\lk_\Delta(e)) = f(\Gamma_e)$ for a simplicial complex $\Gamma_e$ with a Boolean decomposition, then $\Gamma *_{\Gamma_e * v} u$ and $\Ast_\Gamma(p) *_{\Gamma_e} u$ also have Boolean decompositions for any ``Boolean'' vertices $u$, $v$, and $p$ of $\Gamma$. \\
	 		
	 		\item Suppose that $\Delta$ is a simplicial complex satisfying one of the following properties:
	 		
	 		\begin{itemize}
	 			\item $\Delta$ is vertex decomposable and flag.
	 			
	 			\item $\Delta$ is a balanced, vertex decomposable, flag simplicial complex.
	 		\end{itemize}
	 		
	 		Then, these properties are also satisfied by $\Ast_\Delta(v)$ and $\lk_\Delta(v)$ in the respective cases. Also, the existence of Boolean decompositions for simplicial complexes $\Gamma_1$ and $\Gamma_2$ such that $h(\Ast_\Delta(v)) = f(\Gamma_1)$ and $h(\lk_\Delta(v)) = f(\Gamma_2)$ would imply that one exists for $\Gamma \coloneq \Gamma_1 *_{\Gamma_2} u$ for a new vertex $u$ (noting that $h(\Delta) = f(\Gamma)$). \\ 
	 		
	 	\end{enumerate}
	 \end{cor}
	 


	 \color{black}

	 \color{black}
	
	\color{black}

\end{document}